%% file: asme2e.tex
\documentclass[twocolumn,10pt]{asme2e}
\usepackage{amsmath,amssymb,amsfonts}
\usepackage{graphicx}
\graphicspath{ {./image/} }
\usepackage{unicode}
\newtheorem{proposition}{Proposition}
\usepackage[dvipsnames]{xcolor}
\usepackage{bbold}
\usepackage{xcolor}
\usepackage{hyperref}

\usepackage[ruled,vlined]{algorithm2e}

\confshortname{the ASME 2020}
\conffullname{Dynamic Systems and Control Conference \\
              DSCC2020}
\confdate{4-7}
\confmonth{October}
\confyear{2020}
\confcity{Pittsburgh, PA}
\confcountry{USA}

\papernum{DSCC2020-3139}

\title{Attack-resilient observer pruning for path-tracking control of Wheeled Mobile Robot}

\author{Yu Zheng, Olugbenga Moses Anubi
    \affiliation{
	FAMU-FSU College of Engineering\\
	Tallahassee, FL 32310, USA\\
    Email: yz19b@fsu.edu, oanubi@fsu.edu
    }	
}

\begin{document}

\maketitle    

\begin{abstract}
{\it Path-tracking control of wheeled mobile robot (WMR) has gained a lot of research attention, primarily because of its wide applicability -- for example intelligent wheelchairs, exploration-assistant remote WMR. Recent increase in remote and autonomous operations\slash requirements for WMR has led to more and more use of IoT devices within the control loop. Consequently, providing interfaces for malicious interactions through false data injection attacks (FDIA). Moreover, optimization-based FDIAs have been shown to cause catastrophic consequences in feedback control systems while by-passing any residual-based monitoring system. Since these attacks target system measurement process, this paper focuses on the problem of improving the resiliency of dynamical observers against FDIA. Specifically, we propose an attack-resilient pruning algorithm which attempts to exclude compromised channels from being processed by the observer. The proposed pruning algorithm improves attack-localization precision to $100\%$ with high probability, which correspondingly improves the resiliency of the underlying UKF to FDIA. The improvements due to the developed resilient pruning-based observer is validated through a numerical simulation of a two-layer path-tracking control platform of differential-driven wheeled mobile robot (DDWMR) under FDIA.} 
\end{abstract}

\begin{nomenclature}
The following notations and definitions are used throughout the whole paper: ${\mathbb R}, {\mathbb R}^n, {\mathbb R}^{n \times m}$ denote the space of real numbers, real vectors of length $n$ and real matrices of $n$ rows and $m$ columns respectively. ${\mathbb R}_{+}$ denotes positive real numbers. Normal-face lower-case letters $(e.g.\hspace{1mm} x \in {\mathbb R})$ are used to represent real scalar, bold-face lower-case letter $(e.g.\hspace{1mm} \mathbf{x} \in {\mathbb R}^n)$ represents vectors, while normal-face upper case $(e.g.\hspace{1mm} X \in {\mathbb R}^{n \times m})$ represents matrices. $\mathbb{1}$ represents all-ones vector. Let $\mathcal{T} \subseteq \{1,\dots, n\}$, then for a matrix $X \in {\mathbb R}^{n \times m}$, $X_{\mathcal{T}} \in {\mathbb R}^{|\mathcal{T}| \times m}$ is the sub-matrix obtained by extracting the rows of $X$ corresponding to the indices in $\mathcal{T}$. $\mathcal{T}^c$ denotes the complement of a set $\mathcal{T}$ and the universal set on which it is defined will be clear from the context. The symbol $\circ$ denotes element-wise multiplication of two vectors and is defined as $\mathbf{z} = \mathbf{x} \circ \mathbf{y}$, where $\mathbf{z}_{i} = \mathbf{x}_{i} \cdot \mathbf{y}_{i}$. The symbol $*$ denotes the convolution operator for vectors. $\textsf{supp}(\mathbf{x})$ donotes the support of the vector $\mathbf{x}$ given by the set $\mathcal{T}=\textsf{supp}(\mathbf{x}) = \{i|\mathbf{x}_i \neq 0\}$. $\textsf{argsort} \downarrow (\mathbf{x})$ denotes a function that returns the sorted indices of vector x in descending order. The space of all square integrable signals is denoted by $\mathcal{L}_2$. The space of all point-wise bounded signals is denoted by $\mathcal{L}_\infty$.
\end{nomenclature}

\section{INTRODUCTION}
\input{introduction}
\section{path-tracking control for DDWMR}\label{sec:path_tracking_control}
\input{control}
\section{False data injection attack}\label{sec:FDIA}
\input{FDIA}

\section{Resilient pruning observer design}\label{sec:Resilient_pruning}
\input{Resilient_observer_design}
\section{Simulation}\label{sec:Simulation}
\input{Simulation}
\section{Conclusion}\label{sec:Conclusion}
\input{Conclusion}
\bibliographystyle{asmems4}
\begin{acknowledgment}
Thanks to Florida State University and the Center for Advanced Power Systems for remote working support on this paper during the COVID-19 outbreak. The authors wish everyone safety in these difficult times.
\end{acknowledgment}
\bibliography{asme2e}
\end{document}

%% file: introduction.tex

Nonholonomic wheeled mobile robots (WMRs) have attracted much attention in the past two decades due to its great mobility and the broad range of applications\cite{roy2015robust}. Quite a lot of researchers have developed  path-tracking controllers for wheeled mobile robots considering nonlinearities \cite{kim1999tracking, oriolo2002wmr, d1992dynamic}, robustness against model uncertainties\cite{dixon2000robust, aguiar2007trajectory}, robustness against noise\cite{cortesao2003kalman,coelho2005path}.
The control strategies depend on the measurements of the robots' velocities and/or location coordinates. However, due to the increasing dependence on IoT devices and wireless communication, the resulting tight coupling of computation, communication and physical components enables malicious agents to inject attacks via the sensors and actuators\cite{pajic2017design}. Consequently, controller would make decision based on attacked measurements or the vehicle would receive malicious control signals. One type of attacks, false data injection attacks (FDIAs), has been shown to be capable of fooling bad data detection (BDD) scheme to compromise the integrity of the state estimator, even with very sparse measurements corruption. This results in false operations of the whole system without any alarm \cite{mo2010false,mo2010fals}. Therefore, it is necessary to develop an attack-resilient observer-based control scheme to mitigate the effect of those attacks. 


 Many authors \cite{pajic2017design, fawzi2014secure, anubi2019resilient} have proposed an $\ell_0$-based resilient state estimators with different modifications or under different scenarios. These estimators have been validated using cruise control of autonomous ground vehicle, electrical power systems, industrial control systems. However, with the exception of \cite{anubi2019resilient}, none of the estimators were validated against large percentage FDIA. Also, in \cite{guo2017exploiting}, a robot intrusion detection system (RIDS) is designed by leveraging physical dynamics of mobile robots. However, the detection engine is a residual-based Chi-square scheme, which is known to be vulnerable to coordinated FDIAs considered in this paper.

Inspired by recent developments in estimation and compressive sensing, we propose a pruning algorithm to mitigate the effect of FDIA on UKF. Consider a linear measurement model under attack:
$$\mathbf{y}=H \mathbf{x} +\mathbf{e},$$
where, $H \in {\mathbb R}^{m \times n}$ is the linear measurement operator, $\mathbf{x} \in {\mathbb R}^n$ is the state vector, $\mathbf{y} \in {\mathbb R}^m$ is the attacked measurement corrupted by a sparse attack vector $\mathbf{e} \in {\mathbb R}^m$. Consequently, attack-resilient estimation is often formulated as a classical error correction problem\cite{fawzi2014secure, anubi2018robust, anubi2019resilient, pajic2017design}: 
\begin{align*}\textsf{Minimize}: \|\mathbf{e}\|_{\ell_0} \hspace{3mm} \textsf{Subject to:} \hspace{2mm} \mathbf{y}=F\mathbf{e},\end{align*}
where $F \in {\mathbb R}^{n \times m}$ is a coding matrix with $n \ll m$ and $FH=\textbf{0}$. It is known\cite{candes2005decoding, candes2006stable} that if the number of attacked nodes is small enough, exact state estimation can be guaranteed by solving the above problem. However, it is shown\cite{pajic2016attack} that exact recovery is unattainable by solving the problem above if more than 50\% of the sensor nodes are attacked. Moreover, the $\ell_0$ optimization problem above is NP-hard and is often relaxed by solving a convex problem if coding matrix satisfies \emph{Restricted Isometry Property} (RIP) \cite{candes2005decoding, candes2008restricted}.

Suppose there is an oracle which gives the exact $\textsf{supp}(\mathbf{e})$ a priori, then the resilient state estimation problem becomes trivial since any decent regression algorithm will be able to recover the states exactly from the non-attacked set. The challenge, however, is that no such oracle exists. Although, there is a host of localization algorithms\cite{mestha2019cyber} designed to serve this purpose, they are always not exact with significant \emph{false positive and false negative rates}. This observation is the central motivation for developing the pruning algorithm. Therefore, the pruning problem to increase the \texttt{signal-to-attack-ratio} of the measurement system using any pre-designed inexact attack localization scheme (subsequently referred to as the \texttt{oracle}). Then the existing least-square based robust estimation algorithms can be implemented for the \emph{pruned} measurements sets to create a resilient estimator. This process requires a certain amount of redundancy in the measurement system. Otherwise, the estimation problem will be rendered under-determined by the pruning process. Quantifying the required redundancy level for a given oracle is beyond the scope of this present work and will be addressed in future work.

Although, there is a lot of work in the literature on resilient Kalman filtering, typical least-square  based  robust  estimator,  mitigating sensors failures, distortion, delay, strong noise interference and more reasons for corrupt signals\cite{wang2014stochastically, mahmoud2007resilient, qu2013optimal}. However, the specific characteristics of attack, unbounded but sparse, make those resilient filters be hard to perform attack-resiliently. To the best of the authors' knowledge, this paper represents one of the earliest approach to prune measurement channels in real-time in order to improve the resiliency of an underlying observer against FDIA.


The rest of paper is organized as follows. In Section~\ref{sec:path_tracking_control}, a two-layer controller is designed, with UKF, to track a reference trajectory with noisy measurement system. In Section~\ref{sec:FDIA}, an optimization-based FDIA algorithm designed to bypass the monitor is also implemented. In Section~\ref{sec:Resilient_pruning}, the channel pruning algorithm is developed and combined with traditional UKF to create a resilient observer. In Section~\ref{sec:Simulation}, simulation results are presented to validate the proposed pruning-based resilient observer. In Section~\ref{sec:Conclusion}, concluding remarks and future directions are given.

%% file: control.tex
In this section, we present a basic two-layer observer-based path tracking controller for a differential-driven wheeled mobile robot (DDWMR). This will be the platform where subsequent pruning algorithm and FDIA are implemented. Figure.~\ref{DDWMR} shows the schematic of the DDWMR considered in this paper.

\begin{figure}[t]
\begin{center}
\includegraphics[scale=0.5]{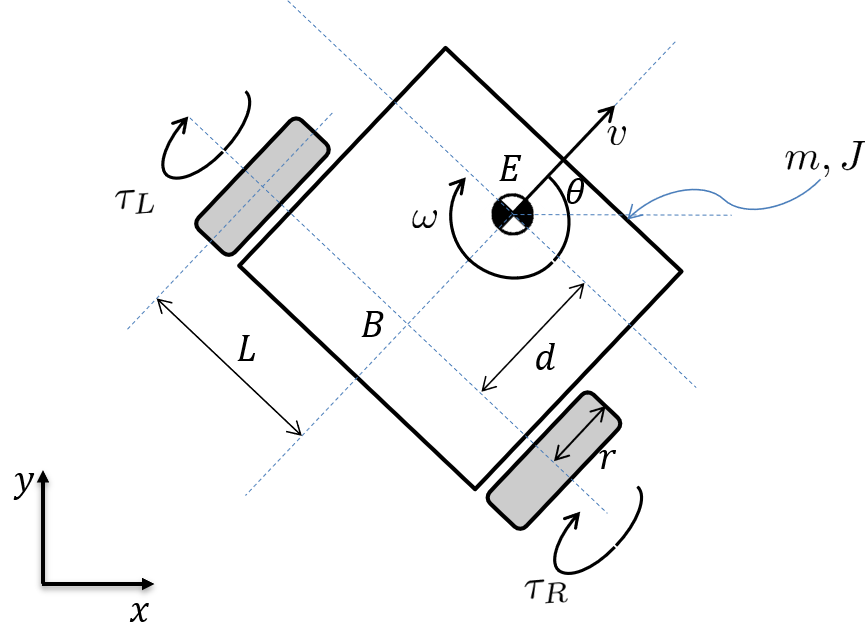}
\end{center}
\caption{Schematic Diagram of the Considered DDWMR Showing Relevant Kinematic and Geometric Quantities }
\label{DDWMR} 
\end{figure}

\begin{figure}[t]
\begin{center}
\includegraphics[scale = 0.35]{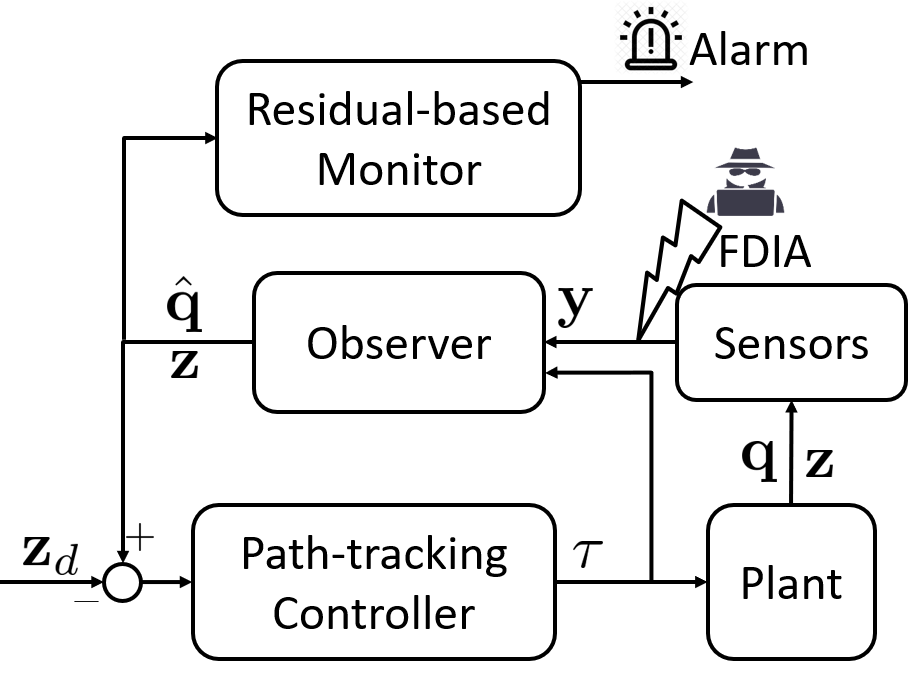}
\end{center}
\caption{Schematic Diagram of the two-layer observer based control system and the attack injection}
\label{2_layer} 
\end{figure} 

The dynamic and kinematic models of DDWMR are given by\cite{dhaouadi2013dynamic}:
\begin{equation}
\begin{aligned}
    &\dot{\mathbf{q}} =  M^{-1} (- D \mathbf{q} + B \mathbf{\tau})+\mathbf{w} \triangleq g(\mathbf{x},\mathbf{u}) +\mathbf{w}\\
    &\begin{bmatrix}\dot{\theta} \\\cdots\\ \dot{\mathbf{z}}\end{bmatrix} =\begin{bmatrix}\begin{array}{cc}0\hspace{1mm}&\hspace{1mm}1\end{array}\\ \cdots \\ \huge C(\theta) \end{bmatrix} \mathbf{q} \triangleq \bar{C}(\theta) \mathbf{q},
\label{model}
\end{aligned}
\end{equation}
where, $\mathbf{q}=\left[v\hspace{2mm} \omega\right]^\top$ is the generalized body velocities vector, $\mathbf{u} \triangleq \mathbf{\tau}=[\tau_R\hspace{2mm}\tau_L]^\top$ is a vector of the wheels torques, and $\mathbf{z} =[x \hspace{2mm} y]^{\top}$ is the task-space position vector, $\mathbf{x}=\left[\theta\hspace{2mm} v\hspace{2mm} \omega\right]^\top$ is defined as a state vector, $\mathbf{w} \sim \mathcal{N}(0,R)$ is the process noise in dynamics. 

The kinematic and dynamical parameters are given by:
$$\begin{aligned}
M&=\begin{bmatrix} m & 0 \\ 0 & m d^2+J \end{bmatrix}, \hspace{2mm}
D=\begin{bmatrix} 0 & -m d\omega \\ m d\omega & 0 \end{bmatrix} \\
B&=\frac{1}{r}\begin{bmatrix} 1 & 1 \\ L & -L\end{bmatrix},\hspace{2mm}
C(\theta) =  \begin{bmatrix}\cos(\theta)& -d \sin(\theta) \\ \sin(\theta) & d \cos(\theta) \end{bmatrix}.
\end{aligned}
$$


Given a reference task-space trajectory $[\theta_d(t) \hspace{2mm} \mathbf{z}_d(t)^\top]^\top$, where $\mathbf{z}_d(t) \in {\mathbb R}^2$ is the corresponding planar Cartesian coordinates of the desired trajectory. We assume that $\mathbf{z}_d(t)$ is continuously differentiable with bounded derivatives, and that all its derivative up to the $2$nd order are known.
Next, consider the tracking error given by
\begin{align}
    \widetilde{\mathbf{e}} = \begin{bmatrix}\theta - \theta_d \\ \mathbf{z}-\mathbf{z}_d\end{bmatrix} = \begin{bmatrix}\mathbf{e}_{\theta} \\ \mathbf{e}_{\mathbf{z}}\end{bmatrix}.
    \label{error_define}
\end{align}
Then, the control law is then designed as:
\begin{equation}
    \mathbf{\tau}=B^{-1}(M \mathbf{u}+ D \mathbf{q}),
    \label{controller}
\end{equation}
where, 
\begin{align*}
    \mathbf{u}= -k_q(\mathbf{q}-\mathbf{q}_d)+\dot{\mathbf{q}}_d - \bar{C}(\theta)^{\top} \widetilde{\mathbf{e}}
\end{align*}
with
$$\begin{aligned}
\mathbf{q}_d&=C^{-1}(\theta)(\dot{\mathbf{z}}_d-k_e \mathbf{e}_{\mathbf{z}}) \\ 
\dot{\mathbf{q}}_d &= -k_e(\dot{C}^{-1}(\theta) \mathbf{e}_{\mathbf{z}} +\mathbf{q})+ C^{-1}(\theta)[ \ddot{\mathbf{z}}_d+(k_e+C(\theta)\dot{C}^{-1}(\theta))\dot{\mathbf{z}}_d]
\end{aligned}
$$
and $k_q$, $k_e$ are positive scalar control gains.

\begin{proposition}
Consider the control law given in (\ref{controller}), if control gains $k_q>0$ and $k_e>0$, then the tracking errors in (\ref{error_define}) converges to zero asymptotically. Moreover, the generalized velocities tracking error $\widetilde{\mathbf{q}}=\mathbf{q}-\mathbf{q}_d$ converges to zero asymptotically with $\dot{\mathbf{z}}_d=C(\theta)\mathbf{q}_d$ satisfied in the limit.
\end{proposition}

\begin{proof}
Consider the candidate Lyapunov function:
\begin{equation}
    V=\frac{1}{2}\| \widetilde{\mathbf{q}}\|^2+\frac{1}{2}\|\widetilde{\mathbf{e}}\|^2
    \label{lyapunov}
\end{equation}

taking the first time derivative and substituting (\ref{model}), (\ref{error_define}), (\ref{controller}) yields 
\begin{equation}
\begin{aligned}
    \dot{V} =& \widetilde{\mathbf{q}}^{\top}\bigg(-k_q\widetilde{\mathbf{q}}-\begin{bmatrix}0\\\mathbf{e}_{\theta}\end{bmatrix}-C(\theta) \mathbf{e}_{\mathbf{z}}\bigg) + \mathbf{e}^{\top}_{\theta} \dot{\mathbf{e}}_{\theta}+\mathbf{e}^{\top}_{\mathbf{z}}\big(C(\theta)\mathbf{q}_d-\dot{\mathbf{z}}_d\big)\\
     =& -k_q \|\widetilde{\mathbf{q}}\|^2 - (\omega-\omega_d)\mathbf{e}_{\theta} - \mathbf{e}^{\top}_{\mathbf{z}}C(\theta)(\mathbf{q}-\mathbf{q}_d)+\mathbf{e}^{\top}_{\theta} \dot{\mathbf{e}}_{\theta} \\
     &+\mathbf{e}^{\top}_{\mathbf{z}}(C(\theta)\mathbf{q}-\dot{\mathbf{z}}_d)\\
    =& -k_q \|\widetilde{\mathbf{q}}\|^2 + \mathbf{e}^{\top}_{\mathbf{z}}(C(\theta)\mathbf{q}_d-\dot{\mathbf{z}}_d)\\
    =& -k_q \|\widetilde{\mathbf{q}}\|^2-k_e\|\mathbf{e}_{\mathbf{z}}\|^2
    \end{aligned}
    \label{derivative}
\end{equation}

This implies that $\dot{V}$ is negative semi-definite, and since $V$ is positive, it follows that $V \in \mathcal{L}_{\infty}$. From (\ref{lyapunov}), it follows that $\widetilde{\mathbf{q}}, \widetilde{\mathbf{e}} \in \mathcal{L}_{\infty}$, which also implies that $\mathbf{e}_{\theta} \in \mathcal{L}_{\infty}$.

Integrating (\ref{derivative}) yields 
$$V-V(0) \leq -\int_0^t \big(k_q \|\widetilde{\mathbf{q}}(\tau)\|^2+k_e\|\mathbf{e}_{\mathbf{z}}(\tau)\|^2\big)d\tau
$$
from which it follows that $\widetilde{\mathbf{q}}, \mathbf{e}_{\mathbf{z}} \in \mathcal{L}_2$. Also, $\dot{\widetilde{\mathbf{q}}}=-k_q \widetilde{\mathbf{q}} - \bar{C}(\theta)^{\top}\widetilde{\mathbf{e}} \in \mathcal{L}_{\infty}$ and $\dot{\widetilde{\mathbf{e}}}=\bar{C}(\theta)\widetilde{\mathbf{q}} -k_e \begin{bmatrix}0 \\ \mathbf{e}_{\mathbf{z}}\end{bmatrix} \in \mathcal{L}_{\infty}$, which implies that $\widetilde{\mathbf{e}}$ and $\widetilde{\mathbf{q}}$ are uniformly continuous. Thus, by Barbalat's Lemma \cite{barbalat1959systemes}, it follows that
$$\widetilde{\mathbf{e}}(t) \rightarrow 0, \widetilde{\mathbf{q}}(t) \rightarrow 0$$
\end{proof}

%% file: FDIA.tex
\newtheorem{theorem}{Theorem}[section]
\newtheorem{definition}{Definition}
\newtheorem{corollary}{Corollary}[theorem]
\newtheorem{lemma}[theorem]{Lemma}
\newtheorem{remark}{Remark}

An attacker can inject false data computed based on a partial or complete knowledge of system model, in order to covertly and accurately change the physical behavior of the plant\cite{7866869}. This section gives the notion of a monitor used in this paper. Based on the monitor, we give a design of FDIA algorithm while assuming an attacker has complete knowledge of system.

For the DDWMR described in previous section, we consider a redundant measurement system of the form:

\begin{equation}\label{eqn:meas_model}
    \mathbf{y}  =  \begin{bmatrix} 1 & 0 \\ 0 & 1 \\ 1/4r & L/4r \\ 1/4r & -L/4r \\ \cos(\theta) & -d \sin(\theta) \\ \sin(\theta) & d \cos(\theta) \end{bmatrix} \cdot \mathbf{q} +\mathbf{v} \triangleq f(\mathbf{x}) + \mathbf{v}
\end{equation}
consisting of both linear and nonlinear components, where $\mathbf{x}=\left[\theta\hspace{2mm} v\hspace{2mm} \omega\right]^\top$ is defined as a state vector, and $\mathbf{v}$ denotes measurement noises.

\begin{definition}[Residual-based Monitor of Horizon $T$]

Based on the closed-loop system in Figure.~\ref{2_layer}, a monitor scheme is any mapping of the form:
$$\Psi_T : \{Y_T, U_T\} \mapsto \{\Psi_1,\Psi_2\}$$
where, $Y_T = \in {\mathbb R}^{m \times T}, U_T \in {\mathbb R}^{l \times T}$ are historical measurements and controlled inputs for $T$ horizon respectively, $\Psi_1 = \{0 (safe), 1(unsafe)\}$ is the first output argument indicating whether or not the data contains attacks, $\Psi_2 = 2^{\{1,2,\cdots, m\}}$ is the second output argument indicating the support of attacks' location.
\end{definition}

 The monitor outputs $\Psi_1 = \{0\} $ for any measurement vector history $\mathbf{Y}_T = [\mathbf{y}_k, \mathbf{y}_{k-1}, \cdots, \mathbf{y}_{k-T+1}]$ and corresponding control history $\mathbf{U}_T=[\mathbf{\tau}_{k-1} \hspace{2mm} \mathbf{\tau}_{k-2} \cdots \mathbf{\tau}_{k-T}]$ if there exists estimate history $\hat{\mathbf{X}}_T=[\hat{\mathbf{q}}_k, \hat{\mathbf{q}}_{k-1}, \cdots, \hat{\mathbf{q}}_{k-T}]$ such that
$$\begin{aligned} \|\hat{\mathbf{q}}_{j+1}-g(\hat{\mathbf{q}}_{j},\mathbf{\tau}_{j})\| \leq \varepsilon_w, \hspace{0.3cm} j=k-T, \cdots, k-1\\
\|\mathbf{y}_j - f(\hat{\mathbf{q}}_{j})\| \leq \varepsilon_v, \hspace{0.3cm} j=k-T+1, \cdots, k
\end{aligned}$$
where $\varepsilon_w$ and $\varepsilon_v$ are any real numbers related to process noise and measurement noise.

Otherwise, the monitor outputs $\Psi_1 = \{1\} $ and the support of the sparsest attack vector history $E_T = \{\mathbf{e}_k, \mathbf{e}_{k-1}, \cdots, \mathbf{e}_{k-T+1}\}$ such that
$$\begin{aligned}
      \|\hat{\mathbf{q}}_{j+1}-g(\hat{\mathbf{q}}_{j},\mathbf{\tau}_{j})\| \leq \varepsilon_w, \hspace{0.3cm} j=k-T, \cdots, k-1\\
\|\mathbf{y}_j - f(\hat{\mathbf{q}}_{j})-\mathbf{e}_j\| \leq \varepsilon_v, \hspace{0.3cm} j=k-T+1, \cdots, k
\end{aligned}
$$  

After linearizing \eqref{eqn:meas_model} about the operating point $\mathbf{x}_0 = [\theta_0 \hspace{2mm} v_0 \hspace{2mm} \omega_0]^{\top}$, we discretize it using Euler's approximation with a sampling time $T_s$, and iterate forward $T_f$ samples, one obtains:
\begin{equation}
    Y_f = H \mathbf{x}_k+G \mathbf{u}_f+e
    \label{measurement}
\end{equation}
where, $ Y_f = \left[\mathbf{y}_k \hspace{2mm} \mathbf{y}_{k+1} \hspace{2mm} \cdots \hspace{2mm} \mathbf{y}_{k+T_f}\right]^{\top}$,
$$
    H = \begin{bmatrix}
            C_d\\
            C_d A_m\\
            C_d A_m^2 \\
            \vdots\\
             C_d A_m^{T_f}
        \end{bmatrix},  G = T_s\begin{bmatrix} 0 & 0 & \cdots & 0\\ C_dB_m & 0 &\cdots& 0\\ C_dA_mB_m& C_dB_m&\cdots & 0 \\ \vdots & \vdots & \ddots & \vdots \\ C_dA_m^{T_f-1}B_m &C_dA_m^{T_f-2}B_m &\cdots & C_dB_m \end{bmatrix}
$$
with
$$
A_m = I + T_s \cdot \begin{bmatrix} 0 &0&1 \\ 0 &0 &2d \omega_0 \\ 0 & -\frac{md\omega_0}{md^2+J} & -\frac{mdv_0}{md^2+J}\end{bmatrix}, \hspace{2mm} B_m = T_s \begin{bmatrix} 0 \\ M^{-1}B \end{bmatrix},$$
$$C_d = \begin{bmatrix} 0 & 1 & 0 \\
                      0 & 0 & 1 \\
                      0 & 1/4r & L/4r \\
                      0 & 1/4r & -L/4r \\
                      -v_0\sin(\theta_0)-d\omega_0\cos(\theta_0) & \cos(\theta_0) & -d\sin(\theta_0) \\
                      v_0\cos(\theta_0)+d\omega_0\sin(\theta_0) & \sin(\theta_0) & -d\cos(\theta_0) \\
      \end{bmatrix}$$

Let $H$ admits the singular value decomposition:
$$H=[U_1 \hspace{2mm} U_2]\begin{bmatrix} \sum_1 \\0\end{bmatrix}V,
$$
where, $U_1 \in {\mathbb R}^{m \times n}$, 
$U_2 \in {\mathbb R}^{m \times (m-n)}$, 
$\sum_1 = \text{diag}(\sigma_1,\sigma_2,\cdots, \sigma_n)$, and  
$V \in {\mathbb R}^{n \times n}$, it is obvious that, the FDIA would pass the monitor if the attack vector $\mathbf{e}$ is defined such that the attack measurement $\mathbf{y}_a$ is in the \emph{range space} of the observation matrix $H$ ($ = \textsf{Range}(U_1)$ ).
Consequently, the FDIA is generated by solving the optimization problem:
\begin{equation}
    \begin{aligned}
\textsf{Maximize} \quad & \|{U_{1_\mathcal{T}}}_{.}^{\top} \mathbf{y}_A\| ^2 \\
     \textsf{Subject to} \quad & \|{U_{2_\mathcal{T}}}_{.}^{\top} \mathbf{y}_A \|^2 \leq \alpha
\end{aligned}
\label{optimization}
\end{equation}
for a given support $\mathcal{T}$ of attack locations under upper bound of percentage of attack injection, and $\alpha$ is a threshold value related to observation matrix $H$ and monitor's threshold $\varepsilon_v$.



%% file: Resilient_observer_design.tex
 Data-driven attack localization algorithms \cite{abbaszadeh2019automated,sabbah2006application} are effective ways of achieving resiliency under FDIA. However, it is challenging to correctly locate all attacked nodes due to the fundamental inexactness associated with data-driven algorithms. In this section, we propose a pruning algorithm to  improve the accuracy of localization algorithms. The underlying philosophy is that if the measurement set is sufficiently redundant, a subset with reduced attacked percentage can be obtained by systematically pruning the measurement set. If the attack percentage is reduced to 0, the pruned measurement set is then used with UKF to produce an improved resilient state estimation under FDIA.

  Let the unknown actual support of safe measurements be $\mathcal{T}^{c} = supp(\mathbb{1}-\mathbf{e})$ with an indicator vector $\mathbf{q}$ given,  element-wise, as:
 \begin{equation}
     \mathbf{q}_{i}=\left\{
     \begin{array}{lr}
     1 &\quad\text{if}\hspace{0.2cm} i \in \mathcal{T}^{c} \\
     0 &\quad\text{otherwise}
     \end{array}
     \right.
 \end{equation}
 Suppose the localization oracle gives an estimated support $\hat{\mathcal{T}^{c}}$ with $\hat{\mathbf{q}}$. Then, the disagreement between the oracle and the actual support can be modeled as:
 \begin{equation}
     \mathbf{q}_{i}=\epsilon_{i}\hat{\mathbf{q}_{i}}+(1-\epsilon_{i})(1-\hat{\mathbf{q}_{i}}),
\label{uncertainty_model}
 \end{equation}
 where $\epsilon_{i}$ depicts the agreement between the estimated and actual support as follows:
 \begin{equation}
     \epsilon_{i}=\left\{
     \begin{array}{ll}
     1 & \text{ if } \hat{\mathbf{q}_{i}} = \mathbf{q}_{i}\\
     0 & \text{ if }\hat{\mathbf{q}_{i}} = 1- \mathbf{q}_{i}
     \end{array}
     \right.
\label{agreement}
 \end{equation}
 It is assumed that $\epsilon_i \sim \mathcal{B}(1,p_i)$, where $p_i$ is given by the true positive rate from the oracle ROC statistics. Moreover, one can see that $\sum_{i=1}^{m} \epsilon_{i}$ is Poisson-Binomially distributed with probability mass function given by:
  \begin{equation}
      \textsf{Pr}\bigg(\sum_{i=1}^{m} \epsilon_{i} = k-1 \bigg)=\mathbf{r}(k), k=1,\cdots,m+1
\label{PMF}
  \end{equation}
  where \cite{5461658}, $\mathbf{r} = \displaystyle\prod_{i=1}^{m} P_{i} \cdot \begin{bmatrix}
\frac{1-P_{1}}{P_{1}}\\
1 \end{bmatrix} * \begin{bmatrix}
\frac{1-P_{2}}{P_{2}}\\
1 \end{bmatrix} * \cdots * \begin{bmatrix}
\frac{1-P_{m}}{P_{m}}\\
1 \end{bmatrix}$, $\mathbf{r} \in \mathbb R_{m+1}$.

Thus, given a reliability level $\eta \in (0,1)$, we define the maximum integer $l_{\eta} \leq m$ for which oracle will correctly localize at least $l_{\eta}$ nodes with a probability of at least $\eta$:
\begin{equation}
\begin{aligned}
    l_{\eta} &= \max \bigg\{k \bigg|  \textsf{Pr} \bigg(\sum_{i=1}^{m} \epsilon_{i} \geq k \bigg)\geq \eta \bigg\} \\
    & = \max\bigg\{k \bigg|1-\sum_{i=1}^{k+1} \mathbf{r}_{i} \geq \eta \bigg\} \\
    &= \max \bigg\{k \bigg|\sum_{i=1}^{k+1} \mathbf{r}_{i} \leq 1-\eta \bigg\}
\label{reliable_number}
\end{aligned}
\end{equation}

Next, we retain the oracle output for the first $l_{\eta}$ most trusted nodes. Let $\mathbf{s} \in [0,1]^{m}$ be a vector of confidence values for the oracle output for each node, then a robust support can be estimated as:
\begin{equation}
    \hat{\mathcal{T}^{c}_{\eta}} = \hat{\mathcal{T}^{c}} \cap \big\{\textsf{argsort} \downarrow (\mathbf{p} \circ \mathbf{s})\big\}_{1}^{l_{\eta}}.
\label{pruning_alg}
\end{equation}

\begin{remark}
(\ref{reliable_number}) and (\ref{pruning_alg}) constitute a pruning scheme for which the resulting $\hat{\mathcal{T}^{c}_{\eta}}$ excludes all attacked channel with a probability larger than $\eta$,
$$ \textsf{Pr}\{\hat{\mathcal{T}^{c}_{\eta}} \cap \mathcal{T} = \emptyset\} \geq \eta.
$$
\end{remark}

Following the pruning operation, the safe measurement model used for a UKF is:
\begin{equation}
    \mathbf{y}_{\hat{\mathcal{T}^{c}_{\eta}}} = f_{\hat{\mathcal{T}^{c}_{\eta}}}(\mathbf{x}) + \mathbf{v}_{\hat{\mathcal{T}^{c}_{\eta}}}.
\label{safe_measurement_model}
\end{equation}

Following standard unscented transformation\cite{julier1997new}, we use $2n+1$ sigma points to approximate the $n$-dimensional normally distributed state $\mathbf{x}$ with assumed mean $\bar{\mathbf{x}}$ and covariance $P_{\mathbf{x}}$ as follows:
$$
 \begin{aligned}
 &\chi_0= \overline{\mathbf{x}}\\
 &\chi_i= \overline{\mathbf{x}} + (\sqrt{(\lambda +n)P_{\mathbf{x}}})_i,\hspace{2mm} i=1,\cdots,n\\
 &\chi_{i+n}= \overline{\mathbf{x}} + (\sqrt{(\lambda +n)P_{\mathbf{x}}})_{i-n}, \hspace{2mm} i=n+1,\cdots,2n
 \end{aligned}$$
The corresponding weights for the sigma points are then given by:
 $$
 \begin{aligned}
 &W_0^m = \lambda/(n+\lambda), W_0^c=W_0^m+(1-\alpha^2+\beta)\\
 &W_i=1/2(L+\lambda)
 \end{aligned}$$
where, $\lambda=\alpha^2(n+\kappa)-n$ represents how far the sigma points are away from the state, $\kappa \geq 0, \alpha \in (0,1]$, and $\beta=2$ is the optimal choice for Gaussian distribution.

Assume $\mathbf{x}_{k-1} \sim \mathcal{N}(\bar{\mathbf{x}}_{k-1}, P_{\mathbf{x},k-1})$, sigma points update through time in sequence with the pruning measurement model in (\ref{safe_measurement_model}). Moreover, according to the corresponding weight, we can predict the new time step state and calculate the new error covariances between the sigma points and the predicted state as follow:
$$
\begin{aligned}
\mathcal{X}_k^{\star}&=g(\mathcal{X}_{k-1},L_{\hat{\mathcal{T}^{c}_{\eta}}}(\hat{\mathbf{x}_k}))\\
\hat{\mathbf{x}}^{-}_k&=\sum_{i=0}^{2n}W_i\mathcal{X}_{k,i}^{\star}\\
\hat{P}_{\mathbf{x},k}&=\sum_{i=0}^{2n}W_i(\mathcal{X}_{k,i}^{\star}-\hat{\mathbf{x}}_k)(\mathcal{X}_{k,i}^{\star}-\hat{\mathbf{x}}_k)^T+R \\
 \mathcal{Y}_{k,\hat{\mathcal{T}^{c}_{\eta}}}&=f_{\hat{\mathcal{T}^{c}_{\eta}}}(\mathcal{X}_{k})
\end{aligned}$$

Next, the measurements and Kalman gains updates are given by:
$$\begin{aligned}
\hat{\mathbf{y}}_{k,\hat{\mathcal{T}^{c}_{\eta}}}&=\sum_{i=0}^{2n}W_i\mathcal{Y}_{(k,i),\hat{\mathcal{T}^{c}_{\eta}}}\\
\hat{P}_{\mathbf{y},k}&=\sum_{i=0}^{2n}W_i(\mathcal{Y}_{(k,i),\hat{\mathcal{T}^{c}_{\eta}}}-\hat{\mathbf{y}}_{k,\hat{\mathcal{T}^{c}_{\eta}}})(\mathcal{Y}_{(k,i),\hat{\mathcal{T}^{c}_{\eta}}}-\hat{\mathbf{y}}_{k,\hat{\mathcal{T}^{c}_{\eta}}})^T+Q\\
\hat{P}_{\mathbf{x}\mathbf{y}}&=\sum_{i=0}^{2n}W_i(\mathcal{X}_{k,i}^{\star}-\hat{\mathbf{x}}_k)(\mathcal{Y}_{(k,i),\hat{\mathcal{T}^{c}_{\eta}}}-\hat{\mathbf{y}}_{k,\hat{\mathcal{T}^{c}_{\eta}}})^T\\
\mathbf{K}_k&=\hat{P}_{\mathbf{x}\mathbf{y}}\hat{P}_{\mathbf{y},k}^{-1}\\
\overline{\mathbf{x}}_{k}&=\hat{\mathbf{x}}_k+\mathbf{K}_k(\mathbf{y}_{k,\hat{\mathcal{T}^{c}_{\eta}}}-\hat{\textbf{y}}_{k,\hat{\mathcal{T}^{c}_{\eta}}}), \hspace{0.2cm} P_{\mathbf{x},k}=\hat{P}_{\mathbf{x},k}-\mathbf{K}_k\hat{P}_{\mathbf{y},k}\mathbf{K}_k^T
\end{aligned}$$
where, $Q$ and $R$ are the measurement and process noise covariance matrices respectively.

In order to numerically verify that the robust support generated by (\ref{pruning_alg}) can achieve $100\%$ localization with a probability of at least $\eta$, we implemented the pruning localization algorithm in a numerical simulation with time-varying FDIAs. The results Figure.~\ref{numerical_simulation} shows that the algorithm achieves $100\%$ localization even for reliability setting $\eta = 0.5$! When the reliability is set to just $0.1$, this algorithm misses only two attacked measurement nodes.

\begin{figure}[t]
\begin{center}
\includegraphics[scale = 0.45]{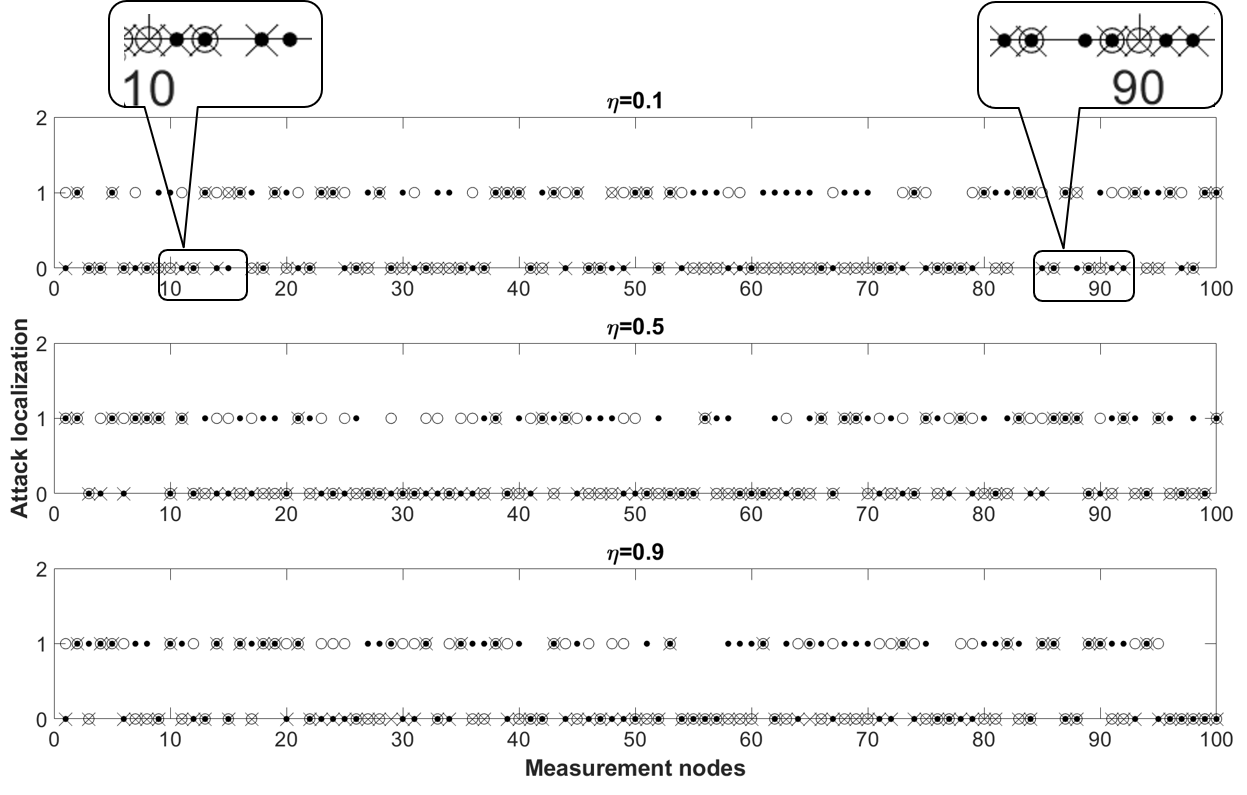}
\end{center}
\caption{Numerical Simulation of pruning algorithm with time-varying FDIAs with $\eta =0.1$, $\eta =0.5$ and $\eta =0.9$ (\textbf{cross} $\rightarrow$ pruning, \textbf{circle} $\rightarrow$ oracle, \textbf{dots at 0} $\rightarrow$ attacked nodes, the perfect result is all dots at $0$ are covered) }
\label{numerical_simulation} 
\end{figure}

%% file: Simulation.tex
In this section, numerical simulation is carried out for DDWMR using three observer strategies under FDIA and the resulting path tracking performance and estimated inner states are compared. The observers compared are: (1) Only UKF, (2) UKF combine directly with the oracle and (3) the proposed pruning-based UKF.
For the path-tracking control system, the control gains are set as $k_1=k_2=10$. The nominal performance of the control system with UKF in an attack-free setting is shown in Figure~\ref{Control_result}. It is seen that the control system, together with UKF, performs well when measurement contains no attack.
\begin{figure}[t]
\begin{center}
\includegraphics[scale = 0.45]{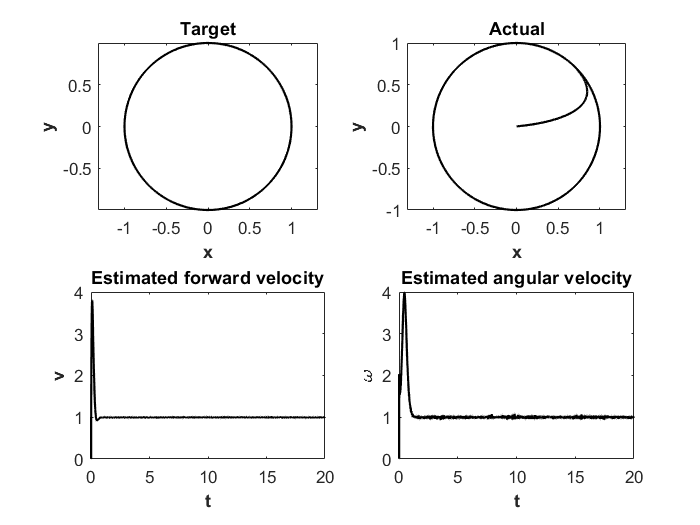}
\end{center}
\caption{Path-tracking and state estimation results of the proposed control system without attacks}
\label{Control_result} 
\end{figure}
Next, a FDIA is implemented and the generated attack vector is added to the system measurements. The oracle is simulated based on the uncertainty model in (\ref{uncertainty_model}) with defined true positive rate $r_p = 0.6$ and confidence for each node localization $s=0.5$. Localization results were then generated to match the specified ROC statistics. The pruning algorithm is implemented with $\eta = 0.8$. The codes for simulation can be found in \url{https://github.com/ZYblend/Resilient-Pruning-Observer-against-False-Data-Injection-Attacks}.

Figures~\ref{path},~\ref{v} and~\ref{w} show the comparison of the performance of three observer strategies under FDIA: "\emph{only UKF}", "\emph{UKF with machine learning}" and "\emph{pruning observer}. The results show that robot cannot track the trajectory under FDIA without any localization and pruning operation, and the estimated dynamic states has very large deviation from the true states. With the oracle, due to the uncertainty, the tracking path is very oscillatory although not as bad as with UKF alone. However, with the proposed observer, the robot was able to track the reference path very closely and smoothly.

\begin{figure}[t]
\begin{center}
\includegraphics[scale = 0.45]{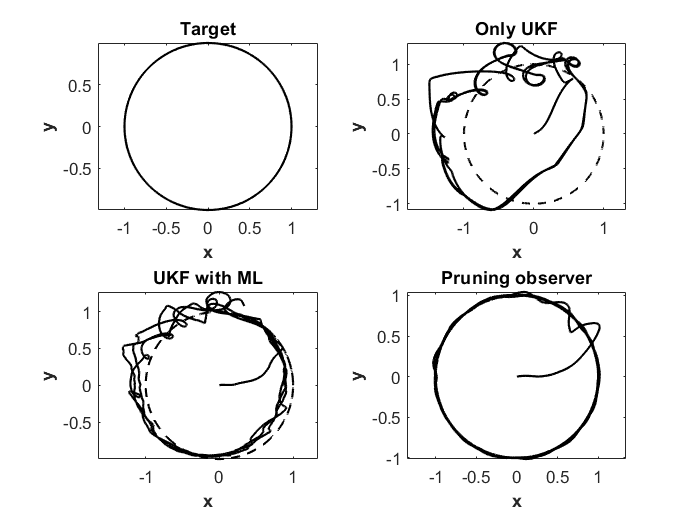}
\end{center}
\caption{A comparison of path tracking results. The proposed pruning observer-based control scheme is able to focus robot to track better, while UKF cannot handle the attacks and machine learning cannot smooth the trajectory. (dot line: reference trajectory, solid line: actual path)}
\label{path} 
\end{figure}

\begin{figure}[t]
\begin{center}
\includegraphics[scale = 0.45]{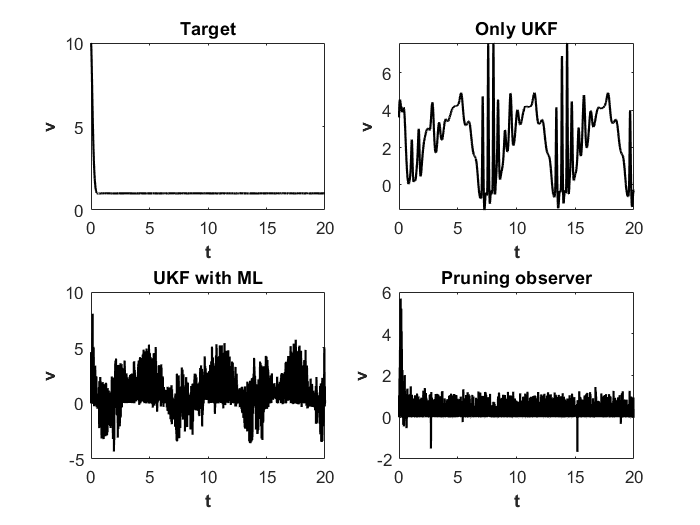}
\end{center}
\caption{A comparison of estimated forward velocity. The proposed pruning observer gives more stable and accurate estimation.}
\label{v} 
\end{figure}

\begin{figure}[t]
\begin{center}
\includegraphics[scale = 0.45]{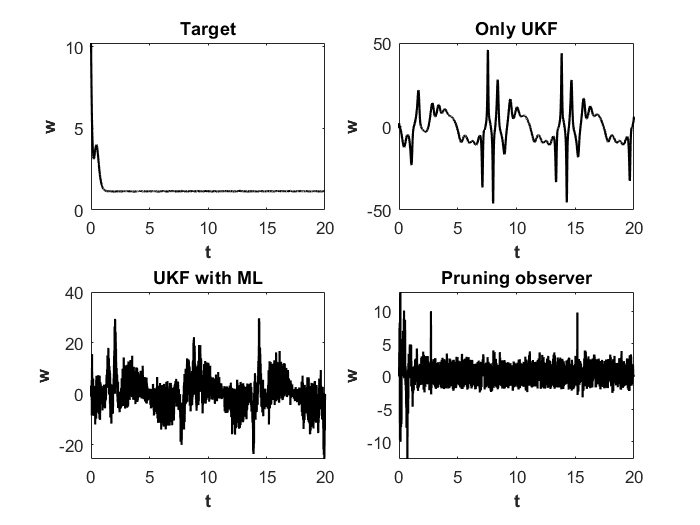}
\end{center}
\caption{A comparison of estimated angular velocity. The proposed pruning observer gives more stable and accurate estimation.}
\label{w} 
\end{figure}

\begin{figure}[t]
\begin{center}
\includegraphics[scale = 0.45]{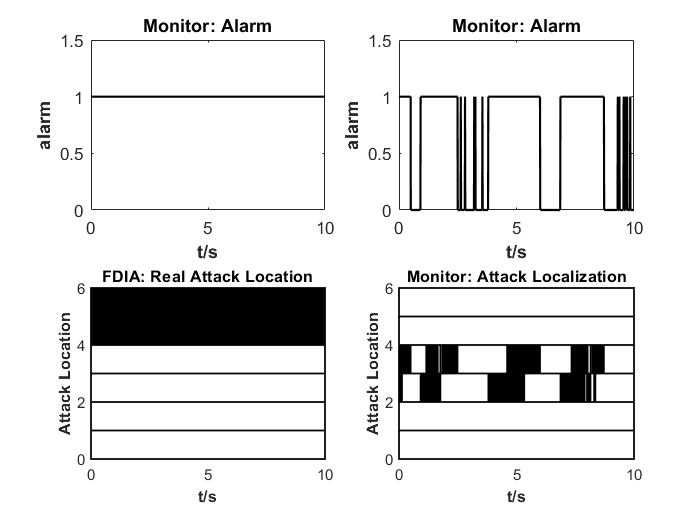}
\end{center}
\caption{FDIAs are localized wrongly by residual-based monitor}
\label{monitor} 
\end{figure}

%% file: Conclusion.tex
In this paper, an attack-resilient path tracking control scheme for wheeled mobile robot under an optimization-based FDIA was designed. The main contributions include: (1) Stable path-tracking control system for DDWMR, (2) Optimization-based FDIA for DDWMR, and (3) The pruning-based observer design using UKF as the underlying observer. It was shown that the proposed pruning-based observer significantly improves the signal-to-attack ratio such that the UKF is able to resiliently estimate the state of the DDWMR even when portion of the sensor measurements were subject to an FDIA. Although this paper shows how promising the resiliency boosting through pruning algorithm is, the results presented only represent the initial stages of this development. Hence there are several open problems that need to be addressed. We name a few:
\begin{enumerate}
    \item As with other resilient observers, the pruning-based resilient observer relies heavily on the inherent redundancy in the measurement system \cite{zhang2006framework}. However, there is no systematic way to quantify the level of redundancy required given any oracle. With $\ell_1$-based methods, the RIP property partly provide answers to this question. What would be interesting to see is how much of a relaxation do we get on the RIP requirements by including pruning? Partial answer to this question can be found in \cite{vaswani2010modified}. We plan to expand on the results as it applies to this problem.
    \item It would be beneficial to see some results on the potential gain by combining pruning and $\ell_1$-based methods.
    \item There are indications from this paper that it is possible to combine pruning directly with the update laws of Kalman filtering algorithms. In future, we will develop a systematic way to achieve this.
    \item We plan to generalize and identify the salient properties for a class of oracles that would combine well with a given underlying estimator.
\end{enumerate}